\documentclass[12pt,a4paper]{article}
\usepackage[breaklinks,pdfencoding=unicode,backref=page]{hyperref}
\usepackage[anythingbreaks]{breakurl}
\usepackage[british]{babel}
\usepackage[UKenglish]{isodate}
\usepackage{amssymb, amsmath, authblk, calligra, mathrsfs, color, fancyhdr, enumitem, mathtools, amsthm, graphicx, array, tikz-cd, subfiles}
\usepackage[toc,page]{appendix}
\usepackage[a4paper, top=3.5cm, bottom=3cm, left=3.0cm, right=3.0cm]{geometry}
\usepackage[all]{xy}
\newtheorem{theorem}{Theorem}
\newtheorem{lemma}[theorem]{Lemma}
\newtheorem{proposition}[theorem]{Proposition}
\newtheorem{corollary}[theorem]{Corollary}

\theoremstyle{definition}
\newtheorem{definition}[theorem]{Definition}
\newtheorem{remark}[theorem]{Remark}
\newtheorem{example}[theorem]{Example}
\newtheorem{algorithm}[theorem]{Algorithm}

\newcommand{\commentaar}[1]{}
\providecommand{\C}{}
\renewcommand{\C}{\ensuremath{\mathbb{C}}}
\newcommand{\R}{\ensuremath{\mathbb{R}}}
\newcommand{\Q}{\ensuremath{\mathbb{Q}}}
\newcommand{\Z}{\ensuremath{\mathbb{Z}}}

\newcommand{\F}{\ensuremath{\mathbb{F}}}

\DeclareMathOperator{\ShHom}{\mathscr{H}\text{\kern -3pt {\calligra\large om}}\,}

\newcommand{\Spec}{\mathop\mathrm{Spec}}

\newcounter{nootje}
\DeclareFontFamily{U}{wncy}{}
\DeclareFontShape{U}{wncy}{m}{n}{<->wncyr10}{}
\DeclareSymbolFont{mcy}{U}{wncy}{m}{n}
\DeclareMathSymbol{\Sh}{\mathord}{mcy}{"58}

\usepackage{enumitem,calc}

\let\OLDthebibliography\thebibliography
\renewcommand\thebibliography[1]{
  \OLDthebibliography{#1}
  \setlength{\parskip}{0pt}
  \setlength{\itemsep}{0pt plus 0.3ex}
}

\newcommand\blfootnote[1]{%
  \begin{NoHyper}
  \renewcommand\thefootnote{}\footnotetext{#1}%
  \addtocounter{footnote}{-1}%
  \end{NoHyper}
}
\begin{document}
\selectlanguage{british}

\title{Computing torsion for plane quartics\\ without using height bounds}
\author{Raymond van Bommel}
\affil{{\small Massachusetts Institute of Technology \& University of Bristol}}

\cleanlookdateon
\maketitle

{\bf Abstract.} We describe an algorithm that provably computes the rational torsion subgroup of the Jacobian of a curve without relying on height bounds. Instead, the strategy is to find upper bounds for the torsion subgroup using reduction modulo primes, and searching for torsion points, not just over $\Q$ but also over small number fields, until the two bounds meet. Both complex analytic and Chinese remainder theorem based methods are used to find such torsion points. The method has been implemented in \texttt{Magma} for plane quartic curves over $\Q$ with a rational point and used to provably compute the rational torsion subgroup for more than 98\% of Jacobians of curves in a data set due to Sutherland consisting of 82240 plane quartic curves. 

\blfootnote{{\bf Keywords:} Abelian varieties, Galois representations, Birch-Swinnerton-Dyer conjecture, Jacobians, Curves}
\blfootnote{{\bf Mathematics Subject Classification (2020):}~11G10, 11F80, 11G40, 11G30, 14H40}

\providecommand{\G}{}
\renewcommand{\G}{\mathrm{Gal}(\overline{\Q}/\Q)}
\newcommand{\Asp}[1]{\mathrm{AGSp}(#1, \F_{\ell})}
\newcommand{\Sp}{\mathrm{GSp}}
\newcommand{\GSp}[1]{\mathrm{GSp}(#1, \F_{\ell})}
\newcommand{\Frobp}{\mathrm{Frob}_p}
\newcommand{\Fix}{\mathrm{Fix}}

\section{Introduction}

In the~1920s, Mordell and Weil proved that for abelian varieties over a number field~$K$ the group of rational points is finitely generated, \cite{Mordell22,Weil29}. In particular, the rational torsion subgroup is finite. The torsion conjecture asserts that there are only finitely many possible torsion groups, when the dimension of the abelian variety and the degree~$[K:\Q]$ are fixed, \cite{SchSch}. The conjecture has been proved in the case of elliptic curves, first over~$\Q$ by Mazur, \cite{Mazur78}, and finally for all number fields by Merel, \cite{Merel96}. The exact determination of which groups can occur, when $K \neq \Q$, and which of them occur infinitely often, is still an active area of research. In higher dimensions, even in the case of abelian surfaces, no upper bound is known for the rational torsion subgroup of an abelian variety over~$\Q$.

The torsion subgroup is also of importance for the Birch and Swinnerton-Dyer conjecture, \cite{BSD}.
The order of the rational torsion subgroup is not only one of the terms appearing in the formula, the action of the absolute Galois group on the torsion is essential for the $L$-function, the central object in the conjecture.
Indeed, the torsion is used to define the Tate module and subsequently the~$L$-function of an abelian variety over a number field is defined using the action of Frobenius on this Tate module.

In this present paper, we consider the question of explicitly computing the rational torsion subgroup in case the abelian variety is the Jacobian of a curve defined over~$\Q$. Recently, such a computation has been done by M\"uller and Reitsma for hyperelliptic curves of genus~3, \cite{MullerReitsma}. For small genus, the torsion is typically computed by doing some form of an exhaustive search. The N\'eron-Tate height~$\widehat{h}$, or canonical height, of torsion points is known to be~0. Then one chooses a na\"ive height~$h$, finds a bound~$|h - \widehat{h}| < c$, and subsequently uses the fact that the na\"ive height of a torsion point is at most~$c$ to find all of them, \cite{Stoll02, Stoll17}. Such height bounds are known and relatively small for genus~1, 2, and~3, but especially in genus~3 and higher the enumeration of all rational points up to these height bounds can still be challenging. Even though methods to compute canonical heights are available, even for non-hyperelliptic curve, \cite{Genus3Heights}, these methods do not give rise to a way to efficiently enumerate points of bounded height. Therefore, we would like to advocate an alternative approach which has been used in an ad hoc fashion in the past and has also been used to acquire information on the Galois representation. The approach could be compared with that in \cite{Mascot20} which is used to compute the whole Galois representation.

The approach uses the fact that, in practice, the torsion points actually seem to have a way smaller na\"ive height than the bound given by the height bounds, and are often not too hard to find. However, the problem then still remains to prove that one has found the complete rational torsion group. For this purpose, one could use the following lemma.

\begin{lemma}\label{lemma:reduction_inj}
Let~$K$ be a number field, let~$\mathfrak{p}$ be a prime of~$\mathcal{O}_K$ over the prime number~$p$, and let~$A$ be an abelian variety over~$K$. Suppose that~$A$ has good reduction at~$\mathfrak{p}$ and that~$A_{\mathfrak{p}}$ is its reduction. Moreover, suppose that~$p > e(\mathfrak{p}/p) + 1$, where~$e(\mathfrak{p}/p)$ is the ramification index of~$\mathfrak{p}$ over~$p$. Then the natural reduction map $$A(K)[n] \longrightarrow A_{\mathfrak{p}}(\mathcal{O}_K/\mathfrak{p})$$ is injective for any integer~$n$ such that~$p \nmid n$. 
\end{lemma}

\begin{proof}
In \cite[Appendix]{Katz81}, there is a proof using formal groups.
Alternatively, one can also use \cite[Thm.\ 3.4.3]{Raynaud} to show that the closure of any subgroup scheme~$(\Z / p\Z)_K \hookrightarrow A[p] \subset A$ inside a N\'eron model~$\mathcal{A} / \mathcal{O}_{K,\mathfrak{p}}$ of~$A$ must be isomorphic to~$(\Z/p\Z)_{\mathcal{O}_{K,\mathfrak{p}}}$, which also proves the injectivity.
\end{proof}

In many cases, it does not suffice to only use~$K = \Q$ in this lemma to find sharp upper bounds for the torsion subgroup.
For elliptic curves, just considering the greatest common divisor of the different orders~$|A_p(\F_p)|$, one can only deduce the existence of an elliptic curve isogenous to~$A$ with that many torsion points, see for example \cite{Katz81}.
One could refine the method by not just considering the orders~$|A_p(\F_p)|$, but actually the group structure of~$A_p(\F_p)$, see for example \cite[Ex.\ 1.3]{MullerReitsma} but even that is not enough.
 For example, it is easy to construct a hyperelliptic curve with nonrational~2-torsion points~$P_2$, $P_3$, and~$P_6$ such that~$P_j$ is defined over~$\Q(\sqrt{j})$ for~$j = 2,3,6$. In this case, for any prime~$p$ of good reduction, at least one of the three points will reduce to a point defined over~$\F_p$, causing the reduction map to never be surjective on the~$\Q$-rational~2-torsion points. In this case, we say that the abelian variety has a {\em fake torsion point}.

The solution that we propose is to also search for torsion points defined over number fields of small degree to account for the nonsurjectivity of the reduction map.
Later in this paper, we will introduce an algebraic and an analytic method to find such nonrational torsion points.
The algebraic method, Algorithm~\ref{algo:algebraic}, constructs nonrational torsion points from their reductions modulo different medium sized auxiliary primes, without assuming anything about the field of definition of the torsion points, using a technique that we call algebraic reconstruction (see Subsection~\ref{subsect:algrec}).
The analytic method, Algorithm~\ref{algo:analytic}, numerically inverts the Abel-Jacobi map to find new torsion points.
Altogether, this leads to the following algorithm to provably compute the torsion subgroup.

\begin{algorithm}\label{algo:main}
{\em Input: a curve $C$ defined over $\Q$.

\vspace{-0.5cm} Output: the rational torsion subgroup of $J = \mathrm{Jac}(C)$.}

\vspace{-0.5cm}\begin{itemize}[labelwidth=\widthof{{\bf Step~2.}},leftmargin=!]
\item[{\bf Step~1.}] Find a set of small primes~$p_1, \ldots, p_s$ satisfying the conditions of Lemma~\ref{lemma:reduction_inj}. For each~$p_i$ compute~$N_i \coloneqq |J_{p_i}(\F_{p_i})|$ (see {\bf Step~2} and {\bf Step~3} of Algorithm~\ref{algo:algebraic}), where~$J_{p_i}$ is the reduction of~$J$ modulo~$p_i$.
\item[{\bf Step~2.}] For each divisor~$\ell \mid \mathrm{gcd}(\{N_i\})$, start with~$T_\ell = \{0_J\}$, and keep applying {\bf Step~3} until the condition in {\bf Step~4} is satisfied.
\item[{\bf Step~3.}] Use Algorithms~\ref{algo:algebraic} and~\ref{algo:analytic} with input~$K = T_\ell \cap J(\Q)$ to add new possibly nonrational torsion points to~$T_\ell$.
\item[{\bf Step~4.}]
Check if all~$\ell$-power torsion can be explained by~$T_\ell$, i.e., if for each~$\ell$-power torsion point~$P \in T_\ell \cap J(\Q)$ there exists a prime~$p_i$ such that every~$Q_{p_i} \in J_{p_i}(\F_{p_i})$ with~$\ell \cdot Q_{p_i} = P \bmod p_i$ is the reduction of some element of~$T_\ell$ at some prime lying above~$p_i$ whose ramification index is at most~$p_i-2$.
\item[{\bf Step~5.}] When the conditions in~{\bf Step~4} are satisfied for each~$\ell$, output the group~$T$ generated by the rational torsion points of all the~$T_\ell$.
\end{itemize}
\end{algorithm}

\begin{theorem}
The output of Algorithm~\ref{algo:main} is correct.
\end{theorem}

\begin{proof}
Suppose there exists a torsion point~$Q \in J(\Q)$ that is not contained in~$T$.
Without loss of generality, we may assume that the order of~$Q$ is a power of a prime~$\ell$.
By Lemma~\ref{lemma:reduction_inj},~$\ell$ must divide all~$N_i$ and hence~$\mathrm{gcd}(\{N_i\})$.
Without loss of generality, we may assume that~$P \coloneqq \ell \cdot Q$ lies in~$T_\ell \cap J(\Q)$.
Now, by the condition in~{\bf Step 4}, there is another nonrational torsion point~$Q'$ in~$T_\ell$ whose reduction modulo a prime lying above~$p_i$ equals that of~$Q$.
This contradicts the injectivety of the reduction map, as shown in Lemma~\ref{lemma:reduction_inj}.
\end{proof}

{\bf Outline.} In Section \ref{sect:prelims}, we discuss the background needed for this approach: the Weil pairing, Weil polynomials, the Newton-Raphson method, and different ways to do Jacobian arithmetic. In Section \ref{sect:faketorsion}, we study the phenomenon of fake torsion points. The core section of this paper in which we explain our actual methods to find torsion points over number fields is Section \ref{sect:methods}. In the final Section \ref{sect:results}, we talk about the computation of the torsion groups in a data set consisting of 82240 plane quartic curves, \cite{Database3}. The implementation of our method in \texttt{Magma} (\cite{Magma}) for plane quartics can be found at \cite{code}.

{\bf Notation.} Throughout this text, $C$ denotes a smooth projective plane quartic curve over~$\Q$, the symbol~$p$ denotes a prime number, $J$ is the Jacobian of~$C$, and~$C_p$ and~$J_p$ are the reduction of~$C$ and~$J$, respectively, modulo~$p$, when~$p$ is a prime of good reduction. 

{\bf Acknowledgements.} The author has been supported by \href{https://simonscollab.icerm.brown.edu/}{the Simons Collaboration} on Arithmetic Geometry, Number Theory, and Computation (Simons Foundation grant 550033) and by C\'eline Maistret's Royal Society Dorothy Hodgkin Fellowship.

The author wishes to thank Edgar Costa, Maarten Derickx, Bjorn Poonen, David Roe, and Andrew Sutherland for useful discussions that helped to improve the method and article, and for their help running the parallel computation on the servers of the Simons Collaboration at the Massachusetts Institute of Technology. The author also wishes to thank the anonymous referees for the valuable suggestions which also led to improvements of the article.

\section{Preliminaries}\label{sect:prelims}

\subsection{Weil pairing}

We will recall the definition and properties of the Weil pairing. For more background we refer the reader to \cite[Sect.~III.8]{SilvermanBook} or \cite[Chap.~8]{MumfordAV}.

For abelian varieties~$A$ over a field~$K$, the {\em Weil pairing} is a bilinear pairing
$$w \colon A(\overline{K})[n] \times A^{\vee}(\overline{K})[n] \to \mu_n\!\left(\overline{K}\right),$$
wher~$n$ is an integer not divisible by the characteristic of~$K$, and where~$\overline{K}$ is an algebraic closure of~$K$.
In particular, in the case of a Jacobian~$J$ of a curve, when the theta divisor on~$J$ induces a principal polarisation~$J \to J^{\vee}$, we get a pairing $$w \colon J(\overline{K})[n] \times J(\overline{K})[n] \to \mu_n\!\left(\overline{K}\right).$$
This pairing, which we will also call the Weil pairing, is symplectic, i.e., alternating and nondegenerate.

Now consider the case where~$K$ is a number field and where~$\mathfrak{p}$ is a prime of residue characteristic~$p \nmid n$, together with~$J$ satisfying the conditions of Lemma \ref{lemma:reduction_inj}. Then the Weil pairing is compatible with the reduction map modulo~$\mathfrak{p}$, i.e., the following diagram is commutative.

\[
\xymatrix{
J(\overline{K})[n] \times J(\overline{K})[n] \ar[r] \ar[d]	&\mu_n\!\left(\overline{K}\right) \ar[d]	\\
J_{\mathfrak{p}}(\overline{\F}_p)[n] \times J_{\mathfrak{p}}(\overline{\F}_p)[n] \ar[r]	&\mu_n\!\left(\overline{\F}_p\right)
}
\]

Moreover, the action of the absolute Galois group~$G_K \coloneqq \mathrm{Gal}\!\left(\overline{K}/K\right)$ respects the Weil pairing, i.e., $$\sigma(w(x,y)) = w(\sigma(x), \sigma(y))$$ for all~$x,y \in J\!\left(\overline{K}\right)\![n]$ and all~$\sigma \in G_K$. In particular the action of~$G_K$ must factor through the general symplectic group~$\Sp(J\!\left(\overline{K}\right)\![n], w)$ through elements with similitude character the~$n$-th cyclotomic character over~$K$. In the case~$n = \ell$ is prime, the group~$\Sp(J\!\left(\overline{K}\right)\![n], w)$ can be identified with the classical general symplectic group~$\Sp(2g, \F_{\ell})$, where~$g$ is the dimension of~$J$.

\subsection{Weil polynomials}
\label{subsect:Lpoly}

We will recall the definition and properties of the Weil polynomial. For more background, we refer the reader to \cite[Sects.\ III.7 and V.2]{SilvermanBook} and \cite{TatePaper}.

Let~$A$ be an abelian variety over~$\Q$, and let~$p$ be an odd prime of good reduction. Then its reduction~$A_p$ is an abelian variety over~$\F_p$ and for any prime~$\ell \neq p$, we can consider the Tate module~$V_\ell \coloneqq \Q_{\ell} \otimes_{\Z_{\ell}} \lim_{n} A_p[\ell^n](\overline{\F}_p)$. The Frobenius element in~$\mathrm{Gal}(\overline{\F}_p/\F_p)$ acts on this~$\Q_{\ell}$-vector space of dimension~$2 \mathrm{dim}(A)$. Its characteristic polynomial~$P_{A_p}$, the Weil polynomial, has coefficients in~$\Z$, is independent of the choice of~$\ell$, and has the property that~$\# A_p(\F_p) =P_{A_p}(1)$.

When~$J = \mathrm{Jac}(C)$ is the Jacobian of a curve and~$p$ is a prime of good reduction of the curve, the polynomial~$P_{J_p}$ can be computed by computing the characteristic polynomial of Frobenius on the \'etale cohomology group~$H^1_{\mathrm{\acute{e}t}}(C, \Q_{\ell})$. It is feasible to compute~$P_{J_p} \bmod p$ for~$p$ of size about~$10^6$ in several minutes using algorithms described in \cite{CostaPhD}, which generalises point counting algorithm for elliptic curves (see for example~\cite{Schoof1,Schoof}) although not in polynomial time in~$\log{p}$, but rather in time $\mathcal{O}(p)$, see \cite{CostaPhD}. Therefore, we will restrict ourselves to primes of the aforementioned size.

\subsection{Newton-Raphson method and precision}
\label{sect:NewtonRaphson}

For part of our computation, we will use complex valued numerical computations in order to try to find algebraic torsion points in~$J(\overline{\Q})$. For this reason, we will briefly recall the numerical methods that we use, their stability, speed of convergence, and the loss of precision that might occur. More background on the methods can be found in many of the textbooks available on numerical analysis.

The main numerical method that we use is the Newton-Raphson method. The method attempts to find a zero for a holomorphic function~$f \colon \C \to \C$ by starting with some initial guess~$x_0 \in \C$ for the root and iteratively computing the next approximation~$x_{n+1} = x_n - \frac{f(x_n)}{f'(x_n)}$. If~$x_0$ is close enough to some simple root~$r$ of~$f$, then the sequence~$(x_n)$ will converge quadratically to~$r$.

However, when~$r$ is a root of multiplicity at least~2 (or in practice also when~$f$ has two simple roots very close to one another) problems may arise. As~$x_n$ gets closer to~$r$, the denominator~$f'(x_n)$ will get close to~0, requiring us to use a lot more precision to reliably compute the fraction~$\frac{f(x_n)}{f'(x_n)}$. In principle, the method would still converge, but the following example explains why it is inevitable that we lose precision.

\begin{example}
Suppose we are looking for a root of~$x^2 - 2x + 1$ close to the starting point~$x_0 = 1.1$, and we are doing computations with~1000 digits of precision. Then we might end up finding the approximate root~$\widetilde{r} = 1 - 10^{-500}$. As~$\widetilde{r}^2 - 2\widetilde{r} + 1 = 10^{-1000}$, this is a root within the precision of our computation, even though the actual root~$r = 1$ lies at distance~$10^{-500}$. We lost half of our digits of precision.

If we now continue using~$\widetilde{r}$ instead of~$r$ and say that we have to find a root of~$x^2 - 2x + \widetilde{r}$, then we could end up with the approximate root~$1 + 10^{-250}$ instead of the intended solution~$r = 1$. So the loss of digits in the process can accumulate, as the output of a root finding algorithm can only be expected to have precision~$\sim \sqrt[n]{\varepsilon}$ given input with precision~$\sim \varepsilon$ around a root of multiplicity~$n$.
\end{example}

The Newton-Raphson can also be used for multivariate functions~$\C^n \to \C^n$. One has to replace~$f'(x_n)$ by the Jacobian matrix~$J$ of the function evaluated at~$x_n$. Here, problems arise when~$J(r)$ has an eigenvalue of zero, in which case the inverse~$J(x_n)^{-1}$ either does not exist, if $J(x_n)$ also has zero as an eigenvalue, or the computation of this inverse becomes numerically unstable as there is an eigenvalue very close to zero.

\subsection{Jacobian arithmetic}

For hyperelliptic curves, points on the Jacobian are typically represented using Mumford coordinates, \cite{Cantor87}, which gives a unique way to represent each curve. For general curves Khuri-Makdisi, \cite{KM04,KM07,KM18}, developed ways to represent points on the Jacobian and do arithmetic. In \cite{FlonOyonoRitzenthaler} Flon, Oyono, and Ritzenthaler describe a method specifically tailored to nonhyperelliptic genus~3 curves.

Many of these methods have been designed with the goal of implementing fast arithmetic over finite fields, which has potential applications in cryptography. Often these methods make assumptions on the curve that are easy to satisfy over finite fields, but not over number fields. For the purpose of our computation, we used two different methods to representing points on the Jacobian and do arithmetic. Even though these methods are most likely not state of the art in terms of their efficiency, we will describe them to inform the reader about the representations we used.

From now on $C$ is a smooth plane quartic curve over $\Q$ and $J$ is its Jacobian.

\subsubsection{Over exact fields}
\label{subsect:jacexact}

Our goal will be to reconstruct points on the Jacobian from their reductions modulo different primes~$p$. For this reason, we would like represent points on the Jacobian in such a way that it has the following properties.

\begin{definition}\label{def:alpha-beta}
For this representation of points on the Jacobian, we define the two properties $(\alpha)$ and $(\beta)$ as follows:\\[-1.2cm]
\begin{itemize}\itemsep0pt
\item[$(\alpha)$] the representation of each point is unique;
\item[$(\beta)$] for a prime~$p$ of good reduction, and any point~$P \in J(\Q)$, there is a way to reduce the representation of~$P$ modulo~$p$, and for all but finitely many of the primes~$p$ this reduced representation is the unique representation for the reduction~$\overline{P} \in J_p(\F_p)$ of~$P$ modulo~$p$.
\end{itemize}
\end{definition}

We will assume throughout that~$C$ has a~$\Q$-rational point.
This is satisfied by at least 99\% of the curves in \cite{Database3}.
The method would work essentially the same for a~$\Q$-rational divisor of degree~1 (e.g., the curve~$x^4 + y^4 + z^4 - 5xz^3$, which has no points over~$\F_5$ and hence no points over~$\Q$, but does have the~$\Q$-rational divisor~$\sum (\sqrt[4]{-1} : 1 : 0) - \sum (\sqrt[3]{5} : 0 : 1)$, where the sum is taken over the conjugates).
In case~$C$ does not have~$\Q$-rational divisors of odd degree, there will be theoretical problems as not every~$\Q$-rational divisor class can be represented by a~$\Q$-rational divisor, which poses serious challenges.
One would probably want to look at the linear algebra methods as described in \cite{KM04}.
We will not touch this subject any further and from now on assume the existence of a~$\Q$-rational point~$P$ on~$C$.

Let~$D$ be any degree~0 divisor~$D$ on~$C$. Then we know there is an integer~$m \geq~0$ such that~$h^0(D + mP) = 1$. Indeed for~$m = 0$, this space has dimension at most~1 by Riemann-Roch, for~$m \gg 0$, the space becomes higher dimensional, and each time we increase~$m$ the dimension increases at most by~1. If there are multiple~$m$ such that~$h^0(D + mP) = 1$, we take the smallest one. Then looking at the divisor of any nonzero function~$f \in H^0(D + mP)$, we find a way to represent~$D$ as~$E - mP$, where~$E$ is an effective divisor. 

\begin{lemma}
This representation, consisting of the integer~$m \geq 0$ and an effective divisor~$E$ of degree~$m$ has property~$(\alpha)$ (as in Definition~\ref{def:alpha-beta}), i.e., it is unique.
\end{lemma}
\begin{proof}
Indeed, the existence of a different effective divisor~$E'$ with the property that~$E - mP \sim E' - mP \sim D$ would imply that~$h^0(D + mP) \geq 2$.
\end{proof}

\begin{lemma}
This representation has property~$(\beta)$ (as in Definition~\ref{def:alpha-beta}).
\end{lemma}

\begin{proof}\newcommand{\Cc}{\mathcal{C}}
Let $D$ be a divisor of degree~0 represented as~$E - mP$ with~$m \geq 0$ and~$E$ as above. 
Let~$N$ be the product of the primes of bad reduction of~$C$. Then~$C$ has a smooth model~$\Cc$ over~$\Spec(\Z[1/N])$ and we can take the closures~$\overline{D}$ and~$\overline{P}$ inside~$\Cc$ of~$D$ and~$P$, respectively. We consider the sheaves~$\mathcal{F} \coloneqq \mathcal{O}_{\Cc}(\overline{D} + m\overline{P})$ and~$\mathcal{G} \coloneqq \mathcal{O}_{\Cc}(\overline{D} + (m-1)\overline{P})$. 
We have that~$\mathcal{F}(\Cc_{\Q})$ and~$\mathcal{G}(\Cc_{\Q})$ are~$\Q$-vector spaces of dimensions~1 and~0, respectively. In particular, this implies that for all but finitely many~$p$ the~$\F_p$-vector spaces~$\mathcal{F}(\Cc_{\F_p})$ and~$\mathcal{G}(\Cc_{\F_p})$ have dimensions~1 and~0, respectively, see \cite[Theorem III.12.8, p.\ 288]{Hartshorne} for example. For these primes~$p$, the representation of~$D \bmod p$ has the same~$m$ and uses~$E \bmod p$.
\end{proof}

\begin{remark}
Note that the proof above does not provide a way to determine which primes~$p$ actually satisfy the conditions in property~$(\beta)$.
This does not pose a problem in practice for our application as the primes chosen in {\bf Step 1} of Algorithm~\ref{algo:main} are so large, that they (almost) always seem to satisfy these conditions, at least for the curves in our dataset.
We did not make any attempt to resolve this problem theoretically.
\end{remark}

To add two points~$E_1 - m_1P$ and~$E_2 - m_2P$, one can use standard methods to compute a basis for the Riemann-Roch spaces~$H^0(E_1 + E_2 + mP)$, finding the smallest~$m$ such that this space has dimension 1.
Then we compute the divisor of a non-zero element~$f \in H^0(E_1 + E_2 + mP)$.
This divisor is linearly equivalent to 0 and is of the shape~$E_{1,2} - E_1 - E_2 - mP$ where~$E_{1,2}$ is effective.
This gives rise to the representation~$E_{1,2} - (m + m_1 + m_2)P$ for $E_1 + E_2 - (m_1+m_2)P$.
This can also be done directly using Magma's built-in function \verb+DivisorReduction+.

\subsubsection{Over the complex numbers}
\label{sect:complex-arithmetic}

From this point on, we assume~$C$ to have genus~3.

For a Zariski dense subset of divisor classes in~$J(\C)$, the representation described in the previous Subsection is of the shape~$-3P + E$ with~$E$ effective of degree~3. Indeed, by Riemann-Roch $h^0(D + 3P) \geq 1$, so~$m \leq 3$, while at the same time the space of divisor classes of the shape~$-mP + E$ with~$0 \leq m \leq 2$ and~$E$ effective of degree~$m$ has dimension at most~2 inside the 3-dimensional Jacobian.

Our (potential) torsion points, being special points on the Jacobian, quite regularly have a representation with~$m < 3$. When we are doing numerical computations on~$J(\C)$, this often causes numerical instability for our algorithms. Luckily, there is an abundance of points on~$J(\C)$ and therefore we can use the following alternative presentation for elements of~$J(\C)$.

We represented them as~$Q_1 + Q_2 + Q_3 - P_1 - P_2 - P_3$, where~$P_1,P_2,P_3 \in C(\C)$ are three arbitrary points that are chosen in advance and~$Q_1,Q_2,Q_3 \in C(\C)$ is a triple of points depending on the divisor class. We will now show why we have a practical guarantee that the set~$\{Q_1, Q_2, Q_3\}$ will be unique for any divisor class that we encounter in our computation.

\begin{proposition}\label{prop:uniqrep}
Let~$D \in J(\C)$ be any divisor class. Then for generic~$P_1, P_2, P_3 \in C(\C)$, the class~$D$ has a unique representation~$Q_1 + Q_2 + Q_3 - P_1 - P_2 - P_3$, where two representations are called the same if the~$Q_i$ are the same up to reordering.
\end{proposition}

\begin{proof}
If~$D$ has two such representations, this implies that~$h^0(D+P_1+P_2+P_3) >~1$, as the associated Riemann-Roch space must have two functions with distinct zeros which are therefore linearly independent. The dimension~$h^0(D+P_1+P_2+P_3)$ is upper semicontinuous as a function in~$P_1,P_2,P_3$ by \cite[Theorem III.12.8, p.\ 288]{Hartshorne} applied to the map $f \colon C^4 \to C^3$ projecting away from the first coordinate and the sheaf $\mathcal{F} = \mathop{\pi_1^*} \mathcal{O}(D) \otimes \mathcal{O}(\Delta_1) \otimes \mathcal{O}(\Delta_2) \otimes \mathcal{O}(\Delta_3)$ where $\pi_1 \colon C^4 \to C$ is the projection on the first coordinate, and $\Delta_i = \{(x_0, x_1, x_2, x_3) \in C^4 : x_0 = x_i \}$.
By Riemann-Roch, we know that~$h^0(D+P_1+P_2+P_3) \geq 1$.
If~$D+P_1+P_2+P_3 \sim~3Q$, where~$Q \in C(\C)$ is a non-Weierstra\ss\ point (i.e., a point~$P$ for which~$h^0(mP) = 1$ for all~$0 \leq m \leq g$, see for example~\cite[Chap.~1]{AlgCurves}), then this dimension is equal to~1. In particular, as~$C(\C)^3$ is irreducible, for all but a codimension~1 set of~$(P_1,P_2,P_3) \in C(\C)^3$, the dimension must equal~1 and the representation must be unique.
\end{proof}

\begin{remark}
In fact, it follows that the map
$$\mathop{\mathrm{Sym}^3} C \to J \colon \quad (Q_1, Q_2, Q_3) \mapsto Q_1 + Q_2 + Q_3 - P_1 - P_2 - P_3$$
is smooth and injective outside of a subset of codimension~1. More details and alternative proofs can be found in~\cite[Chap.\ 1]{AlgCurves}.
\end{remark}

The representation also has the following useful property.

\begin{proposition}\label{prop:goodrep}
Let~$D \in J(\C)$ be a nonzero divisor class. Then for generic elements~$P_1,P_2,P_3 \in C(\C)$, the unique representation~$Q_1+Q_2+Q_3 - P_1-P_2-P_3$ for~$D$ has the property that~$\{Q_1,Q_2,Q_3\} \cap \{P_1,P_2,P_3\} = \emptyset$.
\end{proposition}
\begin{proof}
It suffices to show that for generic~$P_1,P_2 \in C(\C)$, the class~$D$ is not equivalent to~$Q_1+Q_2-P_1-P_2$ for any~$Q_1,Q_2 \in C(\C)$. Equivalently, we need to show that~$h^0(D+P_1+P_2) = 0$ generically, as the existence of~$Q_1$ and~$Q_2$ with~$Q_1+Q_2-P_1-P_2-D \sim 0$ is equivalent to the existence of a function~$f$ with~$\mathrm{div}(f) + D + P_1 + P_2$ effective. Suppose that this is not the case, then this dimension must be at least~1 for any choice of~$P_1$ and~$P_2$ by the semicontinuity in \cite[Theorem III.12.8, p.\ 288]{Hartshorne}. In particular, for any distinct~$P_1,P_2,P_3 \in C(\C)$, we now have three ways of representing~$D$: 
$$Q_1+Q_2+P_3 - \sum_i P_i, \qquad R_1+P_2+R_3 - \sum_i P_i, \quad \textrm{and} \quad P_1+S_2+S_3 - \sum_i P_i,$$
for certain~$Q_1,Q_2,R_1,R_3,S_2,S_3 \in C(\C)$. By the uniqueness of the representation, which we may assume for generic~$P_1, P_2, P_3$ by Proposition~\ref{prop:uniqrep}, we now must have that~$\{Q_1,Q_2,P_3\} = \{R_1,P_2,R_3\} = \{P_1,S_2,S_3\} = \{P_1,P_2,P_3\}$. In particular,~$D = 0$, which is a contradiction.
\end{proof}

\begin{corollary}
Let~$D_1, \ldots, D_n \in J(\C)$ be nonzero divisor classes.
Then for all but at most a codimension~1 subset of $(P_1, P_2, P_3) \in C(\C)^3$, the representations $Q_{i,1} + Q_{i,2} + Q_{i,3} - P_1 - P_2 - P_3$ of~$D_i$ satisfy the conditions in Propositions~\ref{prop:uniqrep} and~\ref{prop:goodrep} for all~$i = 1, \ldots, n$.
\end{corollary}

We will from now on assume that all divisor classes that we encounter in the computation of the torsion subgroup satisfy the conditions in Propositions~\ref{prop:uniqrep} and~\ref{prop:goodrep}.

To add two points, we use the following algorithm, which is a modified version of the algorithm in \cite{FlonOyonoRitzenthaler}.

\begin{algorithm}\label{algo:addC}
{\em Input: two triples of points~$Q_1,Q_2,Q_3$ and~$R_1,R_2,R_3$ representing points~$Q = \sum_i Q_i - \sum_i P_i$ and~$R = \sum_i R_i - \sum_i P_i$ on~$J(\C)$.

\vspace{-0.5cm} Output: a triple of points~$S_1,S_2,S_3$ representing the point~$Q+R = \sum_i S_i - \sum_i P_i$.}

\vspace{-0.5cm}\begin{itemize}[labelwidth=\widthof{{\bf Step~2.}},leftmargin=!]
\item[{\bf Step~1.}] Pick (another) random point~$B \in C(\C)$.
\item[{\bf Step~2.}] Find the line~$\ell$ through~$P_1$ and~$P_2$, and compute the residual intersection~$A$ of this line with~$C$, i.e., $A$ is an effective divisor of degree~2 such that~$C$ intersects~$\ell$ in~$P_1+P_2+A$.
\item[{\bf Step~3.}] Find the cubic~$c$ through~$Q_1,Q_2,Q_3,R_1,R_2,R_3,A,$ and~$B$, and compute the residual intersection~$E$ of this cubic with~$C$, i.e., $E$ is an effective divisor of degree~3 such that~$C$ intersects~$c$ in~$\sum_i Q_i + \sum_i R_i +A+B+E$.
\item[{\bf Step~4.}] Find the conic~$n$ through~$B,P_3,$ and~$E$ and compute the residual intersection~$S$ of this conic with~$C$, i.e, $S$ is an effective divisor of degree~3 such that~$C$ intersects~$n$ in~$B+P_3+E+S$.
\item[{\bf Step~5.}] Output the three points~$S_1, S_2,$ and~$S_3$ of which~$S$ consists.
\end{itemize}
\end{algorithm}

\begin{proposition}\label{prop:corr-algo-1}
The output of Algorithm \ref{algo:addC} is correct.
\end{proposition}

\begin{proof}
Consider the rational function~$\frac{c}{\ell n}$. By construction, its associated principal divisor is
\begin{align*}
\left(\frac{c}{\ell n}\right) &=\sum_i Q_i + \sum_i R_i + A + B + E- P_1 - P_2 - A - B - P_3 - E - S\\
&= \sum_i Q_i + \sum_i R_i - \sum_i P_i - \sum_i S_i.
\end{align*}
In particular, we see that~$\sum_i S_i - \sum_i P_i$ is equivalent to~$\sum_i Q_i + \sum_i R_i - 2\sum_iP_i$.
\end{proof}

\begin{remark}\label{rem:num-stable}
To find the residual intersection of a line/conic/cubic with~$f$ numerically, using the root finding algorithms described in Subsection \ref{sect:NewtonRaphson}, it is beneficial for these residual intersections to not have any double points, and no points in common with the rest of the divisor.
In general, we expect the divisor~$P_1+P_2$ to not contain~$A$ and the divisor~$\sum_i Q_i + \sum_i R_i + A + B$ to have no point in common with~$E$, and~$E$ to not have any double points.
Indeed,~$P_1$ and~$P_2$ were chosen generically, so~$\ell$ should generically not be tangent to~$C$.
Similarly,~$Q_1, Q_2, Q_3$ behave like generic points, because~$P_1, P_2, P_3$ are generic.
Note that~$Q$ could equal~$R$.
So~$c$ is either a cubic going through eight generic points on $C$, or a cubic going through five generic points on $C$, of which three with multiplicity 2.
In both cases, the cubic~$c$ has no geometric reason to go through any of these generic points again, or through any other double point.

Because of this, the computation of~$A$ and~$E$ in {\bf Step~2} and {\bf Step~3} is numerically stable and fast. In {\bf Step~4}, there could be an issue when computing the residual divisor~$S$. The divisor~$S$ could namely contain~$P_3$, but according to Proposition \ref{prop:goodrep}, this only happens in the case~$Q+R = 0$. In all other cases, there is generally no double point and our algorithm to compute~$S$ will be numerically stable and fast.
\end{remark}

Another way that we will use to represent points in~$J(\C)$ is by the means of an element in a complex torus~$\C^3 / \Lambda$. The computation of a period lattice~$\Lambda$ and an Abel-Jacobi map~$\iota \colon J(\C) \to \C^3 / \Lambda$ mapping~$Q_1+Q_2+Q_3-P_1-P_2-P_3$ to a corresponding point in the complex torus has been implemented in \texttt{Magma} by Neurohr, see also \cite{Neurohr}. We will also write~$\iota(Q_1, Q_2, Q_3)$ for~$\iota(Q_1+Q_2+Q_3-P_1-P_2-P_3)$.

In order to go back from a point in~$\C^3 / \Lambda$ to a divisor class, we use the following algorithm to invert the Abel-Jacobi map.

\begin{algorithm}\label{algo:inv-AJ}
{\em Input: an element~$x \in \C^3 / \Lambda$.

\vspace{-0.5cm} Output: a triple of points~$Q_1,Q_2,Q_3 \in C(\C)$ such that~$\iota(Q_1,Q_2,Q_3)$ is close to~$x$.}

\vspace{-0.5cm}\begin{itemize}[labelwidth=\widthof{{\bf Step~2.}},leftmargin=!]
\item[{\bf Step~1.}] Pick some integer~$n$. We found that~$n = 14$ worked well in practice for our examples.
\item[{\bf Step~2.}] After picking some local coordinate for~$C$ that is invertible around the points~$P_1$,~$P_2$, and~$P_3$, the map~$\iota$, locally around~$(P_1, P_2, P_3)$, can be considered as an analytic map~$\C^3 \to \C^3$.
Use Newton-Raphson (see Subsection~\ref{sect:NewtonRaphson}) with starting point~$(P_1, P_2, P_3)$ to numerically approximate a solution to~$\iota(Q_{1,n}, Q_{2,n}, Q_{3,n}) = \tfrac{1}{2^n} \cdot x$.
\item[{\bf Step~3.}] Add~$Q_{1,n} + Q_{2,n} + Q_{3,n} - \sum_i P_i$ to itself using Algorithm \ref{algo:addC}. The output of this addition is an approximate solution to~$\iota(Q_{1,n-1}, Q_{2,n-1}, Q_{3,n-1}) = \tfrac1{2^{n-1}} \cdot x$. We then use Newton-Raphson to increase the precision of this solution $(Q_{1,n-1}, Q_{2,n-1}, Q_{3,n-1})$. Decrease~$n$ by~1 and repeat this step until~$n = 0$.
\item[{\bf Step~4.}] Use Newton-Raphson to refine~$(Q_{1,0}, Q_{2,0}, Q_{3,0})$ to the desired precision and output the triple. 
\end{itemize}
\end{algorithm}

The reason for choosing an~$n$ and dividing by~$2^n$ first, is to make sure that the starting point~$(P_1, P_2, P_3)$ in {\bf Step~2} is close enough to the solution for the Newton-Raphson method to actually converge. Because of the assumptions made and by Remark~\ref{rem:num-stable}, the addition in {\bf Step~3} is expected to be numerically stable and to only introduce a moderate error.
This means that the Newton-Raphson method in {\bf Step~3} and {\bf Step~4} can reasonably be expected to converge again.

\subsubsection{Changing base points}

In Subsection~\ref{subsect:jacexact}, we worked with divisors represented as~$E - 3P$ with~$E$ effective of degree~3, while in Subsection~\ref{sect:complex-arithmetic}, we represented the same divisors as~$Q_1 + Q_2 + Q_3 - P_1 - P_2 - P_3$. The following describes an algorithm, over~$\C$ to convert the second representation back into the first representation. It is a modified version of Algorithm~\ref{algo:addC}.

\begin{algorithm}\label{algo:convert}
{\em Input: a triple of points~$Q_1,Q_2,Q_3$ representing a point~$Q$ on~$J(\C)$ as~$\sum_i Q_i - \sum_i P_i$.

\vspace{-0.5cm} Output: a triple of points~$S_1,S_2,S_3$ representing the same point~$Q$ as~$\sum_i S_i - 3P$.}

\vspace{-0.5cm}\begin{itemize}[labelwidth=\widthof{{\bf Step~2.}},leftmargin=!]
\item[{\bf Step~1.}] Pick (another) random point~$B \in C(\C)$.
\item[{\bf Step~2.}] Find the line~$\ell$ through~$P_1$ and~$P_2$, and compute the residual intersection~$A$ of this line with~$C$, i.e., $A$ is an effective divisor of degree~2 such that~$C$ intersects~$\ell$ in~$P_1+P_2+A$.
\item[{\bf Step~3.}] Find the cubic~$c$ through~$Q_1,Q_2,Q_3,A,$ and $B$ with multiplicity~1, and through $P$ with multiplicity~3. Compute the residual intersection~$E$ of this cubic with~$C$, i.e., $E$ is an effective divisor of degree~3 such that~$C$ intersects~$c$ in~$\sum_i Q_i + 3P +A+B+E$.
\item[{\bf Step~4.}] Find the conic~$n$ through~$B,P_3,$ and~$E$ and compute the residual intersection~$S$ of this conic with~$C$, i.e, $S$ is an effective divisor of degree~3 such that~$C$ intersects~$n$ in~$B+P_3+E+S$.
\item[{\bf Step~5.}] Output the three points~$S_1, S_2,$ and~$S_3$ of which~$S$ consists.
\end{itemize}

\begin{proposition}
The output of Algorithm \ref{algo:convert} is correct.
\end{proposition}

\begin{proof}
The proof is similar to that of Proposition~\ref{prop:corr-algo-1}.
Consider the rational function~$\frac{c}{\ell n}$. By construction, its associated principal divisor is
\begin{align*}
\left(\frac{c}{\ell n}\right) &=\sum_i Q_i + 3P + A + B + E- P_1 - P_2 - A - B - P_3 - E - S\\
&= \sum_i Q_i + 3P - \sum_i P_i - \sum_i S_i.
\end{align*}
In particular, we see that~$\sum_i Q_i - \sum_i P_i$ is equivalent to~$\sum_i S_i - 3P$.
\end{proof}

\end{algorithm}

\begin{remark}
What is said in Remark~\ref{rem:num-stable} still holds for {\bf Step~1} and {\bf Step~2}.
The cubic in {\bf Step~3} does have a triple zero at~$P$, but that is not a problem, because we already know what~$P$ is exactly, and we can use derivative functions to check if~$c$ actually goes through~$P$ four (or more) times.
The same can be applied in {\bf Step~4} in case one or more of the~$S_i$ equals~$P$.
Note that the points in~$B + P_3 + E$ behave like generic points, and are generically not equal to one of the~$S_i$.
However, it could still happen that two of the~$S_i$ in {\bf Step~4} are equal to each other, but not equal to~$P$.
This is the only case in which I observed some significant precision loss occurring in the method, with the number of correct digits decreasing by a factor of~3 in the worst case.
\end{remark}

\section{Fake torsion points}\label{sect:faketorsion}

Let~$P \in J(\Q)$ be a point and let~$\ell$ be a prime number. We define $$D_{\ell}(P) = \{ Q \in J(\overline{\Q}) \mid \ell \cdot Q = P \} \quad \textrm{and} \quad D_{\ell,p}(P) = \{Q \in J_p(\overline{\F}_p) \mid \ell \cdot Q = \overline{P}\},$$ for every odd prime~$p \neq \ell$ of good reduction. This is a torsor under the action of~$J[\ell](\overline{\Q})$ or~$J_p[\ell](\overline{\F}_p)$, respectively. We already saw in the introduction that it could happen that the set of~$\Q$-rational points in~$D_{\ell}(P)$ is smaller than any of the sets of~$\F_p$-rational points in~$D_{\ell,p}(P)$. In case this happen, we say that~$P$ has a {\em fake~$\ell$-divisor}.

In order to understand this phenomenon better, one considers the action of the absolute Galois group~$\G$ on~$D_{\ell}(P)$.
Pick a not necessarily rational~$Q_0 \in D_{\ell}(P)$ as base.
Then for any~$g \in G$ the action of~$g$ on~$D_{\ell}$ can be described by
$$Q_0 + x \mapsto g(Q_0) + M \cdot x,$$
where~$M \in \mathrm{GL}(2g, \F_\ell)$.
Because~$g$ has to respect the symplectic form on~$J[\ell]$, the matrix~$M$ actually lies in~$\mathrm{Sp}(2g, \F_\ell)$, and we get a map~$G \to \Asp{2g}$.
Let~$H$ be the image of this map.
In the case~$P = 0$, we actually have that~$H$ is a subgroup of the smaller group~$\GSp{2g}$ since we can pick~$Q_0 = 0$. (The same can be done if we can find a rational $Q_0 \in D_\ell(P)$, but the most interesting case for our method, is the case when such $Q_0$ does not exist.)

For each odd prime~$p \neq \ell$ of good reduction, there is a conjugacy class~$\Frobp$ of~$H$ which describes how Frobenius acts on~$D_{\ell,p}(P)$. Using these, we can exactly determine for which~$H$ the point~$P$ has a fake~$\ell$-divisor.

\begin{proposition}
The point~$P$ has a fake~$\ell$-divisor if and only if for every element~$h \in H$ we have
$$D_{\ell}(P) \supset \Fix(h) \supsetneq \Fix(H) \coloneqq \{x \in D_{\ell}(P) \mid \forall h \in H : h(x) = x\}.$$
\end{proposition}

\begin{proof}
The set~$\Fix(H)$ is exactly the set of~$\Q$-rational points in~$D_{\ell}$(P). For each odd prime~$p \neq \ell$ of good reduction, the set of points in~$D_{\ell}(P)$ reducing to an~$\F_p$-rational point in~$D_{\ell,p}(P)$ is exactly~$\Fix(h)$ for some~$h \in \Frobp$. Because of the Chebotarev density theorem, every conjugacy class will occur as~$\Frobp$ for some odd prime~$p \neq \ell$, which concludes the proof of the proposition.
\end{proof}

Looking at the group~$H$, one can not only determine whether there is a fake torsion point, but also the degrees of the actual torsion points. By enumerating all the appropriate subgroups of~$\Asp{2g}$, we get the following result that shows that in certain cases the nonexistence of rational~$\ell$-divisors of~$P$ can be explained by points of degree at most~12.

\begin{proposition}
Suppose that $g = 3$ and either~$\ell = 2$, or both~$\ell = 3$ and~$P = 0$. Then there exist points~$Q_1, \ldots, Q_k \in D_{\ell}(P)$ such that~$[\Q(Q_i) : \Q] \leqslant~12$ and a prime number~$p$ with the following properties. If~$P \neq~0$, then~$D_{\ell,p}(P) = \{Q_1 \bmod p, \ldots, Q_k \bmod p\}$. If~$P = 0$, then~$D_{\ell,p}(P) = \langle Q_1 \bmod p, \ldots, Q_k \bmod p \rangle$.
\end{proposition}

\begin{proof}
This is a big group theoretic computation, enumerating all the appropriate subgroups of~$\GSp{6}$ or~$\Asp{6}$, and figuring out the degrees of the fake torsion points needed. The code can be found at \cite[\texttt{extra/subgroups.m}]{code}.
\end{proof}

\section{Methods}\label{sect:methods}

In this section, we explain the main result of this paper: two methods to find torsion points over number fields. For the first method, we use the Chinese remainder theorem, taking torsion points modulo~$p_i$ for different primes~$p_i$ and trying to combine them into one torsion point over a number field. For the second method, we use a complex analytic approach, computing a complex approximation of torsion points up to high enough precision to reconstruct them algebraically. One could also imagine a third method, where one uses Hensel lifting to try to construct torsion points using methods from \cite{Mascot20}, but this approach has not been implemented as of now.

\subsection{Algebraic reconstruction}
\label{subsect:algrec}
\newcommand{\pp}{\mathfrak{p}}

Given a rational number~$\alpha = \frac{r}{s}$ and its residue class modulo~$N$ for some suitable~$N \gg \max(r^2, s^2)$, one could wonder if it is possible to construct~$\alpha$ from this residue class. This question has been answered positively in \cite{RatRec1,RatRec2} with a fast algorithm using the Euclidean algorithm.

In this section, we will consider an algebraic number~$\alpha \in \overline{\Q}$, its associated number field~$K = \Q(\alpha)$ and prime ideals~$\pp_1, \ldots, \pp_k$, such that~$v_{\pp_i}(\alpha) \geq~0$ for~$i = 1, \ldots, k$. Then we can reduce~$\alpha$ modulo each~$\pp_i$ and we get finite field elements~$\alpha_i \in \F_{\pp_i}$. The question one can ask now is: can we reconstruct~$\alpha$ from the~$\alpha_i$? We will describe an algorithm that attempts to do this. Even though it is still practical for our purpose, the algorithm is definitely not as efficient as the rational reconstruction algorithm mentioned before.

For each~$i$, let~$p_i$ be the residue class field characteristic of~$\pp_i$, and let~$f_i \in \Z[x]$ be a lift of the minimum polynomial of~$\alpha_i$ over~$\F_{p_i}$. Then we can consider the ideal~$I_i = (f_i, p_i)$ of~$\Z[x]$. The minimum polynomial~$f$ of~$\alpha$ is an element of~$I_i$ for each~$i$ and hence of the intersection~$I \coloneqq \bigcap_i I_i$. The idea of our approach is to find a small element in~$I$.

\begin{algorithm}\label{algo:main}
{\em Input: prime numbers~$p_i$ and polynomials~$f_i$ as described above.

\vspace{-0.5cm} Output: candidate minimum polynomial~$f$ for~$\alpha$.}

\vspace{-0.5cm}\begin{itemize}[labelwidth=\widthof{{\bf Step~2.}},leftmargin=!]
\item[{\bf Step~1.}] Compute a Gr\"obner basis~$G$ for the ideal~$I = \bigcap_i (f_i, p_i) \subset \Z[x]$.
\item[{\bf Step~2.}] Set~$d \coloneqq~1$, the degree for the candidate polynomial~$f$ that we are currently considering.
\item[{\bf Step~3.}] For each~$g \in G$, compute~$B_g^d \coloneqq \{ x^i \cdot g \mid i \in \Z_{\geq~0} \textrm{ such that } \deg(x^i \cdot g) \leq d\}$. Let~$B^d \coloneqq \bigcup_g B_g^d$ and~$\Lambda^d \subset \R\{x^0, \ldots, x^d\}$ be the lattice generated by~$B^d$.
\item[{\bf Step~4.}] Find a short vector~$f \in \Lambda^d$. Compute~$|f|$, the maximum of the absolute values of the coefficients of~$f$.
\item[{\bf Step~5.}] If~$(2|f|)^{d+1}$ is significantly smaller than~$\mathrm{lcm}(\{p_i\})$ and~$f \not \equiv~0 \bmod p_i$ for any~$i$, then return~$f$, otherwise set~$d \coloneqq d+1$ and return to {\bf Step~3}.
\end{itemize}
\end{algorithm}

For {\bf Step~4} of the algorithm one could use any algorithm to find short vectors. In our implementation we used the LLL algorithm by Lenstra, Lenstra, and Lov\'asz, see \cite{LLL}. In {\bf Step~5}, we do a heuristic check to see if the polynomial~$f$ that we are currently considering is small enough. For this purpose, we compare the number of polynomials of the same degree with coefficients of equal or smaller size with the product of the primes~$p$ over which we have information about~$f \bmod p$. If the latter is much greater than the former, this suggests that the polynomial that we are currently considering might be the correct one.

\begin{example}\label{ex:rat-recon}
Suppose that~$k = 2$, $p_1 = 1009$, $p_2 = 1019$, $f_1 = x-55$ and~$f_2 = x-241$.
Then we find~$G = \{x + 635615, 1028171\}$.
For~$d = 1$, the shortest vector that we can find is~$x-392556$, which is too big to pass the test in {\bf Step~5}.
For~$d = 2$, we find the short vector~$x^2 + 2$, which we will output as~$f$.
\end{example}

\subsection{Finding torsion points: the CRT method}

In this section, we will describe how to find torsion points using the Chinese remainder theorem. We assume that~$\ell$ is prime and that we have some~$\ell$-power torsion point~$Q \in J(\Q)$. Our goal is to find points~$R \neq~0$ such that~$\ell R = Q$. In this section, all points will be represented using the representation described in Subsection \ref{subsect:jacexact}. We will first give an outline of the method.

The idea is to look modulo different primes~$p_i$ for points~$R_{p_i}$ having the property that~$\ell \cdot R_{p_i} = Q \bmod p_i$.
These torsion points combined, give a point modulo the product~$\prod p_i$, and using algebraic reconstruction (see Subsection~\ref{subsect:algrec}), we attempt to find a torsion point~$R$ over a number field that reduces to the~$R_i$.
While every torsion point over a number field would theoretically be found eventually using this method, the sheer number of possible combination of torsion points to consider, might make this impractical.
The algorithm, as described below, also uses several tricks in an attempt to keep this number of possible combinations as small as possible, so that we actually manage to get some non-trivial output in practice.

\begin{algorithm}\label{algo:algebraic}
{\em Input: a prime number~$\ell$, a subgroup~$K$ of known torsion points inside~$J[\ell](\Q)$, and an~$\ell$-power torsion point~$Q \in J(\Q)$.

\vspace{-0.5cm} Output: a (possibly empty) list of nonzero points~$R \in J(\overline{\Q})$ such that~$\ell R = Q$.}

\vspace{-0.5cm}\begin{itemize}[labelwidth=\widthof{{\bf Step~10.}},leftmargin=!]
\item[{\bf Step~1.}] Pick some medium size ($\approx~10^6$) auxiliary prime numbers~$p_1, \ldots, p_k$, such that~$C$ has good reduction at these primes.
\item[{\bf Step~2.}] For each~$p_i$, compute the  Weil polynomial~$P_{p_i}$ modulo~$p_i$ of the reduction~$J_{p_i}$ as described in Subsection \ref{subsect:Lpoly}. Using inequalities for the coefficients of~$P_{p_i}$ found in \cite{Haloui}, construct a finite set~$B$ containing all the possible values of~$N_i \coloneqq \#J_{p_i}(\F_{p_i}) = P_{p_i}(1)$.
\item[{\bf Step~3.}] Take a random point~$S \in J_{p_i}(\F_{p_i})$ and use a baby step giant step approach to identify all~$b \in B$ such that~$b \cdot S = 0$. Discard all other elements of~$B$. Repeat this step until~$\#B = 1$, which must mean that~$B = \{N_i\}$.\setcounter{footnote}{0}\footnote{Even though it never happened in our dataset, theoretically, it could happen that $\# B$ never becomes equal to 1. This problem is studied in more detail in~\cite{Shi}, in which there is a resolution to this problem called Step~3. For our purposes, it would always be possible to change the primes chosen in {\bf Step~1} to avoid this.}
\item[{\bf Step~4.}] For each~$p_i$, decompose~$N_i$ as~$\ell^{e_i} \cdot q_i$, where~$q_i$ has no factors~$\ell$. Then generate a bunch of random points~$S$ in~$J_{p_i}(\F_{p_i})$ and compute~$q_i \cdot S$, which is an element of~$J_{p_i}[\ell^\infty]$. Keep finding new points, until there are enough points to generate the~$\ell$-power torsion of~$J_{p_i}(\F_{p_i})$.
\item[{\bf Step~5.}] For each~$p_i$, find the set~$D_{p_i}$ of points~$R_i \in J_{p_i}(\F_{p_i})$ such that~$\ell R_i = Q \bmod p_i$, and compute the image~$K_{p_i}$ of~$K$ inside~$J_{p_i}(\F_{p_i})$. Discard some of the primes~$p_i$ for which the set~$D_{p_i}$ is relatively large.
\item[{\bf Step~6.}] For each finite set~$I \subset \{1, \ldots, k\}$ for which~$D_I \coloneqq \prod_{i \in I} D_{p_i} / K_{p_i}$ is not too large, enumerate representatives~$(R_i)_{i \in I}$ for all elements of~$D_I$ and execute the next three steps for each such element. After finishing that, continue to Step~10.
\item[{\bf Step~7.}] For each~$i \in I$ and~$V \in K$ compute a representation $$R_i + \overline{V} = m_{i,v} \overline{P} + \sum_{m=1}^{-m_{i,V}} R_{i,V,m}, \qquad \textrm{where} \qquad R_{i,V,m} \in C_p(\overline{\F}_p)$$ as in Subsection \ref{subsect:jacexact}. If the multisets~$\{ m_{i,V} : V \in K \}$ are not all equal for the different~$i \in I$, disregard this element of~$D_I$. Otherwise, compute the polynomials~$P_{x,i} = \prod_{m,R} (T - x(R_{i,V,m}))$ and~$P_{y,i} = \prod_{m,R} (T - y(R_{i,V,m}))$ inside~$\F_{p_i}[T]$.
\item[{\bf Step~8.}] Use algebraic reconstruction, as described in Subsection \ref{subsect:algrec}, to try to lift the matching coefficients of the~$P_{x,i}$ and~$P_{y,i}$ for the different~$i \in I$ to elements of a number field. If the coefficients lift, and we get polynomials~$P_x, P_y \in \overline{\Q}[T]$, apply the next step to them.
\item[{\bf Step~9.}] For all possible combinations of the roots of~$P_x$ and~$P_y$ see which ones give points on~$C(\overline{\Q})$. Then try all combinations of~$m$ of these points to see if we can find an~$R \in J(\overline{\Q})$ such that~$\ell R = Q$. Use the Jacobian arithmetic described in Subsection \ref{subsect:jacexact} to verify this.
\item[{\bf Step~10.}] After finishing the loop described in Step 5, output all~$R$ with~$\ell R = Q$ that we found in Step 9 during the computation.
\end{itemize}
\end{algorithm}

The following example, which uses hyperelliptic curves instead of nonhyperelliptic curves and primes that are a bit smaller, demonstrates the method.

\begin{example}
Suppose~$C$ is the hyperelliptic curve~$y^2 = x^7 + x^5 - 4x^3 - 4x$.
Suppose that~$K = \{0\}$,~$\ell = 2$,~$p_1 = 1009$ and $p_2 = 1019$.
Over~$\F_{p_1}$, we find 15 nonzero~2-torsion points, of which one is~$(55,0) - (0,0)$.
Over~$\F_{p_2}$, we find 15 nonzero~2-torsion points, of which one is~$(241,0) - (0,0)$.
Using algebraic reconstruction on the~$x$-coordinates of these divisors, as in {\bf Step~8}, we quickly realise that these are the reductions of the torsion point~$(\sqrt{-2}, 0) - (0,0)$ on~$\mathrm{Jac}(C)(\Q(\sqrt{-2}))$, cf.~Example~\ref{ex:rat-recon}.

Note that~$\mathrm{Jac}(C)$ actually has 8 rational 2-torsion points, so of the~$15 \cdot 15 = 225$ combinations of~2-torsion points in~{\bf Step~6}, only~$7 \cdot (1 \cdot 1) + 4 \cdot (2 \cdot 2) = 23$ of them are the reductions of a common point.
If we had already found some of these~8 rational~2-torsion points, then the combinatorics in {\bf Step~6} would have been better.
\end{example}

{\bf Steps~1} through~{\bf 4} are precomputation steps that only need to be done once for each curve. In most cases, the CRT method was the faster method to find torsion points over number fields. The biggest bottleneck of the method is the combinatorial explosion that can take place in {\bf Steps~6} through~{\bf 9}; the sets~$D_J$ can become very big in cases where there is a lot of fake torsion.

\begin{remark}
Let us give a very rough heuristic analysis of the method above through a fictive example in order to demonstrate this combinatorial explosion.
Suppose that we found two points~$P_1,P_2 \in J(\Q)$ generating a subgroup isomorphic to~$\Z/2\Z \times \Z/2\Z$ and we are looking for a fake 4-torsion point~$Q$ such that~$2Q = P$. Suppose that~$Q$ is defined over a field~$K$ of degree~12 and that its coordinates have minimal polynomials with coefficients of size at most~100.

Note that there are about~$200^{13} \approx~10^{30}$ polynomials of degree~12 with coefficients of size at most~100.
In order to have a good chance to reconstruct~$Q$ from its reductions modulo different primes~$p_i$ in {\bf Step~8}, we want the modulus~$N = \prod_i p_i$ to be significantly greater than~$10^{30}$. That means that we should take at least 6 medium sized primes~$p_1, \ldots, p_6$.

We need~$Q$ to actually be defined over~$\F_{p_i}$ rather than an extension field. The probability of a random prime~$p_i$ having this property is the same as the probability that a random element of the Galois group of~$K$ fixes one of the 12 roots of the minimal polynomial of the coordinates. In the worst case, this probability is $\tfrac1{12}$, meaning that we actually needed to pick about~72 medium sized primes and find the right combination of the 6 primes among the ${72 \choose 6} \approx~10^8$ possibilities.

For each prime~$p_i$ there are at least 4 candidate points in~$D_{p_i}$, namely $\overline{Q}$, $\overline{Q + P_1}$, $\overline{Q + P_2},$ and $\overline{Q + P_1 + P_2}$. Of the at least~$4^6 = 4096$ combinations of points we can pick, only~4 of them can be used to successfully reconstruct a point of degree~12. This part of the combinatorial explosion has been mitigated by the replacement of~$D_{p_i}$ by~$D_{p_i} / K_{p_i}$ in {\bf Step~6}, but there is still the problem that some of the other fake torsion points will show up in some of the~$D_{p_i}$. 
\end{remark}

\subsection{Finding torsion points: the analytic method}

The following analytic method to find torsion points has the advantage that there will be no combinatorial explosion of trying to combine torsion points modulo different primes into a torsion point over a number field. The downside is that we cannot utilise the fact that~$J_{p_i}(\F_{p_i})[\ell^n]$ is typically much smaller than~$J(\C)[\ell^n]$. Recall that we assumed the existence of a point~$P \in C(\Q)$ and that we picked such a point at the start.

\begin{algorithm}\label{algo:analytic}
{\em Input: a prime number~$\ell$, a subgroup~$K$ of known torsion points inside~$J[\ell](\Q)$, and a point~$Q \in J(\Q)$ as described above.

\vspace{-0.5cm} Output: a (possible empty) list of nonzero points~$R \in J(\overline{\Q})$ such that~$\ell R = Q$.}

\vspace{-0.5cm}\begin{itemize}[labelwidth=\widthof{{\bf Step~2.}},leftmargin=!]
\item[{\bf Step~1.}] Choose some~$P_1, P_2, P_3 \in C(\C)$ as in Subsection \ref{sect:complex-arithmetic}. Compute an Abel-Jacobi map~$\iota \colon J(\C) \to \C / \Lambda$ with base points~$P_1,P_2,P_3$ and compute the image~$\iota(Q)$, using the methods in \texttt{Magma} implemented by Neurohr, see~\cite{Neurohr}.
\item[{\bf Step~2.}] Pick an element~$t$ in each class in~$\left(\tfrac1\ell \iota(Q) + \tfrac1\ell \Lambda\right) / \iota(K)$ and apply the following three steps for each element.
\item[{\bf Step~3.}] Use Algorithm \ref{algo:inv-AJ} to find points~$R_1', R_2', R_3' \in C(\C)$ such that~$\iota(R_1', R_2', R_3')$ is close to~$t$. Use Algorithm~\ref{algo:convert} to write~$\iota(R_1', R_2', R_3')$ as~$R_1 + R_2 + R_3 - 3P$ for some~$R_1, R_2, R_3 \in C(\C)$.
\item[{\bf Step~4.}] Compute Mumford-like coordinates for~$R$, i.e., compute the product polynomial~$P_x \coloneqq \prod_i (T - x(R_i))$ in~$\C[T]$ and a polynomial~$P_y$ of degree~2 such that~$P_y(x(R_i)) = y(R_i)$.
\item[{\bf Step~5.}] Use a short lattice vector algorithm to try to find algebraic relations for the coefficients of~$P_x$ and~$P_y$. If this succeeds, reconstruct the corresponding point in~$J(\overline{\Q})$, which we call~$R_t$. 
\item[{\bf Step~6.}] After finishing the loop described in Step~2, output all~$R_t$ with~$\ell R_t = Q$ that we found in Step 5 during the computation.
\end{itemize}
\end{algorithm}

In practice, to recognise torsion points over number fields, we need several hundreds of digits of precision. This together with the sheer number of potential points we need to try (typically~$\ell^6$) makes the method slow in practice and only practical for~$\ell = 2$ or~$\ell = 3$.

\section{Results}\label{sect:results}

The algorithm has been implemented by the author in \texttt{Magma} and is publicly available at \cite{code}. It has been run on a data set consisting of 82240 plane quartic curves found by Andrew Sutherland, see \cite{Database3}. As a result, the rational torsion subgroup has been computed successfully for 81357 of the Jacobians of these curves. The total runtime for this computation, using version 2.25-7 of \texttt{Magma}, was approximately 8 core months and has been done in parallel, using at most 20 GB of memory per process, on a machine of the Simons Collaboration at Massachusetts Institute of Technology having a AMD EPYC 7713 CPU. For each computed torsion group a proof has been stored in the form of a list of primes, and a list torsion points over~$\Q$ and over some number fields which together can be used to prove the completeness of the list of rational torsion points using Lemma \ref{lemma:reduction_inj}. 
The vast majority of the time spent in Algorithms~\ref{algo:algebraic} and~\ref{algo:analytic} was going through a huge number of potential torsion points over number fields.
Hence, after these points have been found, the proofs can be verified significantly faster than it took to construct them.
The proofs are stored in the file \cite[\texttt{extra/proofs.tar.xz}]{code}. In Table~\ref{table:orders}, you can see the 96 different group structures of the torsion groups that we found and how often each of them occurred. In this table, the notation $n_1, n_2, \ldots$ in the top row is referring to the group $\Z/n_1\Z \times \Z/n_2\Z \times \cdots$, and the number on the bottom row indicates how often we found this group.





\begin{table}[ht]
\begin{center}
\begin{tabular}{c|c|c|c|c|c|c|c|c|c|c|c}
1 &2 &3 &4 &2,2 &5 &6 &7 &8 &4,2 &2,2,2 &9 \\ \hline
58702 &8855 &5101 &2404 &291 &1106 &1435 &616 &431 &264 &2 &379
\end{tabular}\\[0.3cm]

\begin{tabular}{c|c|c|c|c|c|c|c|c|c|c|c|c|c|c}
3,3 &10 &11 &12 &6,2 &13 &14 &15 &16 &8,2 &4,4 &4,2,2 &17 &18 &6,3 \\ \hline
73 &214 &51 &324 &58 &42 &130 &78 &21 &67 &35 &6 &7 &90 &37
\end{tabular}\\[0.3cm]

\begin{tabular}{c|c|c|c|c|c|c|c|c|c|c|c|c|c|c}
19 &20 &10,2 &21 &22 &23 &24 &12,2 &6,2,2 &25 &5,5 &26 &27 &9,3 &28 \\ \hline
30 &43 &8 &55 &17 &2 &55 &30 &1 &11 &3 &14 &9 &10 &23
\end{tabular}\\[0.3cm]

\begin{tabular}{c|c|c|c|c|c|c|c|c|c|c|c|c|c}
14,2 &29 &30 &31 &32 &16,2 &8,4 &8,2,2 &4,4,2 &33 &35 &36 &18,2 &12,3 \\ \hline
8 &1 &24 &3 &1 &3 &14 &1 &1 &12 &5 &9 &6 &12
\end{tabular}\\[0.3cm]

\begin{tabular}{c|c|c|c|c|c|c|c|c|c|c|c|c|c|c}
6,6 &38 &39 &40 &20,2 &41 &42 &22,2 &45 &15,3 &48 &24,2 &49 &50 &51 \\ \hline
6 &4 &7 &9 &6 &1 &15 &2 &6 &1 &4 &10 &2 &2 &2
\end{tabular}\\[0.3cm]

\begin{tabular}{c|c|c|c|c|c|c|c|c|c|c|c|c|c}
52 &54 &18,3 &56 &28,2 &57 &60 &30,2 &62 &16,4 &8,4,2 &65 &66 &70 \\ \hline
2 &1 &1 &1 &3 &4 &7 &1 &1 &1 &5 &1 &3 &3
\end{tabular}\\[0.3cm]

\begin{tabular}{c|c|c|c|c|c|c|c|c|c|c}
24,3 &12,6 &75 &15,5 &40,2 &84 &96 &24,4 &14,7 &105 &40,4 \\ \hline
2 &3 &1 &1 &2 &1 &1 &2 &1 &1 &1
\end{tabular}
\end{center}
\caption{Torsion group statistics}
\label{table:orders}
\end{table}

We also kept track of the number of cases in which we needed to find a fake torsion point in order to prove the upper bound for the torsion group. For~3440 of the curves, we needed one or more fake torsion points and in Table~\ref{table:fake-torsion} you can find maximum degrees for these fake torsion points, and how often they occurred.

\begin{table}[ht]
\begin{center}
\begin{tabular}{c|c|c|c|c|c|c|c|c}
none	&deg.\ 2	&deg.\ 3	&deg.\ 4	&deg.\	6	&deg.\	8	&deg.\	9	&deg.\	10	&deg.\	12	\\	\hline
77917	&1386	&191	&478	&52	&70	&1	&10	&1217
\end{tabular}

\end{center}
\caption{Fake torsion statistics: maximum degree of fake torsion points}
\label{table:fake-torsion}
\end{table}

For the majority of the 883 missing plane quartics, the reason that we could not compute their torsion subgroup was the failure to find a rational point on the curve. For 735  of these curves, we could verify the nonexistence of rational points by proving that there are no points over some local field. For the remaining 148 curves, which might give rise to counterexamples for the Hasse principle, we did not attempt to verify the nonexistence of rational points. This problem has recently been studied in more detail in~\cite{BruinCreutz}.

To conclude this section we exhibit an example where we managed to find a torsion point over a degree~12 number field in order to certify the correctness of the computed rational torsion subgroup.

\begin{example}
Consider the smooth plane quartic~$C \colon f = 0$ with
$$f = x^3y - xy^3 + y^4 + x^3z + 2x^2yz + 2xy^2z - y^3z + x^2z^2 + 2xyz^2 + y^2z^2 - 2xz^3 - yz^3 + z^4.$$
Its Jacobian~$J$ modulo~11 has~1772 points, and~$J$ modulo 67 has~274944 points. As the primes~11 and 67 are both primes of good reduction, this implies that the torsion subgroup of~$J$ can have at most order~$\mathrm{gcd}(1772, 274944) = 4$.

Besides~0, we find a second rational torsion point $$T_2 = \left(\tfrac{1}{19}(-10\theta + 4) : -1 : 1\right) + \left(\tfrac{1}{19}(-10\bar{\theta} + 4) : -1 : 1\right) - 2 \cdot (1:0:0),$$ where~$\theta$ and~$\bar\theta$ are zeros of~$x^2 - \tfrac{27}{10}x - \tfrac{1713}{100}$. Looking at the reduction modulo~11, we easily find that there are no other points of order~2. After looking at a lot of different primes and seeing that~$T_2$ has a~2-divisor modulo each of these primes, we suspect that~$T_2$ might have a (fake)~2-divisor.

After about an hour of computation time, our program finds a torsion point~$T_4$ over a degree~12 number field~$K$ defined by adjoining to~$\Q$ a root of $$x^{12} - 5x^{10} - 2x^9 - 20x^8 - 20x^7 + 7x^6 - 50x^5 + 26x^4 - 40x^3 - 58x^2 - 24x - 15.$$
This point satisfies~$2T_4 = T_2$. As the prime 67 splits into four primes of residue degrees 67, 67, $67^2$, and~$67^8$ in the ring of integers of~$K$, the point~$T_4$ explains two of the~2-divisors of~$T_2$ modulo 67. As there are only two~2-divisors of~$T_2$ in~$J \bmod 67$, we conclude that~$T_2$ doesn't have a~2-divisor over~$\Q$, and~$\{0, T_2\}$ is the full torsion subgroup of~$J$.
\end{example}

\end{document}